\documentclass[11pt,reqno,tbtags,a4paper]{amsart}
\usepackage{amssymb}
\usepackage{xpunctuate}
\usepackage{url}
\usepackage[square,numbers]{natbib}
\bibpunct[, ]{[}{]}{;}{n}{,}{,}

\title
{On general subtrees of a conditioned Galton--Watson tree}

\date{9 November, 2020}

\author{Svante Janson}
\thanks{Supported by the Knut and Alice Wallenberg Foundation}
\address{Department of Mathematics, Uppsala University, PO Box 480,
SE-751~06 Uppsala, Sweden}
\email{svante.janson@math.uu.se}
\newcommand\urladdrx[1]{{\urladdr{\def~{{\tiny$\sim$}}#1}}}
\urladdrx{http://www2.math.uu.se/~svante/}

\numberwithin{equation}{section}

\renewcommand\le{\leqslant}
\renewcommand\ge{\geqslant}

\allowdisplaybreaks

\theoremstyle{plain}
\newtheorem{theorem}{Theorem}[section]
\newtheorem{lemma}[theorem]{Lemma}
\newtheorem{proposition}[theorem]{Proposition}

\theoremstyle{definition}

\newtheorem{exampleqqq}[theorem]{Example}
\newenvironment{example}{\begin{exampleqqq}}
  {\hfill\qedsymbol\end{exampleqqq}}

\newtheorem{remarkqqq}[theorem]{Remark}
\newenvironment{remark}{\begin{remarkqqq}}
  {\hfill\qedsymbol\end{remarkqqq}}

\theoremstyle{remark}

\newenvironment{romenumerate}[1][-10pt]{
\addtolength{\leftmargini}{#1}\begin{enumerate}
 \renewcommand{\labelenumi}{\textup{(\roman{enumi})}}%
 }{\end{enumerate}}

\newcounter{oldenumi}
{\setcounter{oldenumi}{\value{enumi}}
\begin{romenumerate} \setcounter{enumi}{\value{oldenumi}}}
{\end{romenumerate}}

\newcounter{thmenumerate}

\newcounter{xenumerate}   

\newcommand\pfitem[1]{\par(#1):}
\newcommand\pfitemx[1]{\par#1:}
\newcommand\pfitemref[1]{\pfitemx{\ref{#1}}}

\newcommand{\refT}[1]{Theorem~\ref{#1}}

\newcommand{\refL}[1]{Lemma~\ref{#1}}

\newcommand{\refR}[1]{Remark~\ref{#1}}

\newcommand{\refS}[1]{Section~\ref{#1}}

\newcommand{\refProp}[1]{Proposition~\ref{#1}}

\newcommand{\refE}[1]{Example~\ref{#1}}
\newcommand{\refEs}[1]{Examples~\ref{#1}}

\newcommand\marginal[1]{\marginpar[\raggedleft\tiny #1]{\raggedright\tiny#1}}
\newcommand\REM[1]{{\raggedright\texttt{[#1]}\par\marginal{XXX}}}
\newcommand\XREM[1]{\relax}

\begingroup
  \divide\count255 by 60
  \multiply\count255 by -60
  \advance\count255 by \time
  \ifnum \count255 < 10 \xdef\klockan{\the\count1.0\the\count255}
  \else\xdef\klockan{\the\count1.\the\count255}\fi
\endgroup

\newcommand{\sumj}{\sum_{j=1}^\infty}

\newcommand{\sumr}{\sum_{r=1}^\infty}
\newcommand{\sumik}{\sum_{i=1}^k}
\newcommand{\sumin}{\sum_{i=1}^n}

\newcommand{\prodik}{\prod_{i=1}^k}

\newcommand\set[1]{\ensuremath{\{#1\}}}

\newcommand\xpar[1]{(#1)}
\newcommand\bigpar[1]{\bigl(#1\bigr)}
\newcommand\Bigpar[1]{\Bigl(#1\Bigr)}

\newcommand\bigsqpar[1]{\bigl[#1\bigr]}

\newcommand\xcpar[1]{\{#1\}}

\newcommand\abs[1]{\lvert#1\rvert}
\newcommand\bigabs[1]{\bigl\lvert#1\bigr\rvert}
\newcommand\Bigabs[1]{\Bigl\lvert#1\Bigr\rvert}

\def\rompar(#1){\textup(#1\textup)}    

\def\xexp(#1){e^{#1}}

\newcommand\ntoo{\ensuremath{{n\to\infty}}}

\newcommand\bmin{\land}

\newcommand\downto{\searrow}
\newcommand\upto{\nearrow}
\newcommand\punkt{\xperiod}    
\newcommand\iid{i.i.d\punkt}    
\newcommand\ie{i.e\punkt}
\newcommand\eg{e.g\punkt}

\newcommand\cf{cf\punkt}
\newcommand{\as}{a.s\punkt}

\newcommand\ii{\mathrm{i}}

\newcommand{\tend}{\longrightarrow}
\newcommand\dto{\overset{\mathrm{d}}{\tend}}
\newcommand\pto{\overset{\mathrm{p}}{\tend}}

\newcommand\lito{\overset{L^1}{\tend}}

\newcommand\Op{O_{\mathrm p}}

\newcommand\bbZ{\mathbb Z}

\newcommand\bbZgeo{\mathbb Z_{\ge0}}

\newcounter{CC}
\newcounter{cc}

\newcommand\E{\operatorname{\mathbb E{}}}
\renewcommand\P{\operatorname{\mathbb P{}}}

\newcommand\Var{\operatorname{Var}}

\newcommand\Po{\operatorname{Po}}
\newcommand\Bi{\operatorname{Bi}}

\newcommand\Ge{\operatorname{Ge}}

\newcommand\gd{\delta}
\newcommand\gD{\Delta}
\newcommand\gf{\varphi}
\newcommand\gam{\gamma}

\newcommand\gs{\sigma}

\newcommand\gss{\sigma^2}

\newcommand\eps{\varepsilon}

\renewcommand\phi{\xxx}  

\newcommand\cT{{\mathcal T}}

\newcommand\indic[1]{\boldsymbol1\xcpar{#1}}

\newcommand\qw{^{-1}}
\newcommand\qww{^{-2}}
\newcommand\qq{^{1/2}}
\newcommand\qqw{^{-1/2}}

\newcommand\intoo{\int_0^\infty}

\newcommand\intpipi{\int_{-\pi}^\pi}

\newcommand\dd{\,\mathrm{d}}
\newcommand\ddx{\mathrm{d}}

\newcommand{\chf}{characteristic function}

\newcommand\rhs{right-hand side}

\newcommand\GW{Galton--Watson}
\newcommand\GWt{\GW{} tree}
\newcommand\cGWt{conditioned \GW{} tree}
\newcommand\GWp{\GW{} process}

\newcommand\xoo{_1^\infty}
\newcommand\xooo{_0^\infty}

\newcommand\nux{n}
\newcommand\nnx{N}
\newcommand\ttt{{\mathbf t}}
\newcommand\ntt{\nnx_{\ttt}}
\newcommand\nut{\nux_\ttt}
\newcommand\nnqr{\nnx_{t_{q,r}}}
\newcommand\nuqr{\nux_{t_{q,r}}}
\newcommand\nttm{N_\ttt^M}
\newcommand\nutm{\nu_\ttt^M}
\newcommand\ctn{\cT_n}
\newcommand\ct{\cT}
\newcommand\df{depth first}
\newcommand\dfo{depth first order}
\newcommand\gdt{\gD(\ttt)}
\newcommand\bm{\mathbf{m}}
\newcommand\bbzzk{\bbZgeo^k}
\newcommand\sumbm{\sum_{\bm}}
\newcommand\ddt{\frac{\ddx}{\ddx t}}
\newcommand\px[1]{{\mathsf{P}_{#1}}}
\newcommand\pk{\px{k}}
\newcommand\pl{\px{\ell}}
\newcommand\tqr{t_{q,r}}
\newcommand\vv{\varpi}
\newcommand\vvl{\vv_\ell}
\newcommand\hT{\widehat T}

\begin{document}

\begin{abstract} 
We show that the number of copies of a given rooted tree in a conditioned
Galton--Watson tree satisfies a law of large numbers under a minimal
moment condition on the offspring distribution.
\end{abstract}

\maketitle


\section{Introduction}\label{S:intro}

Let $\ctn$ be a random \cGWt{} with $n$ nodes, defined by an offspring
distribution $\xi$ with mean $\E\xi=1$,
and let $\ttt$ be a fixed ordered rooted tree. 
We are interested in the number of copies of $\ttt$ as a (general)
subtree of $\ctn$,
which we denote by $\ntt(\ctn)$.
For details of these and other definitions, see \refS{Snot}.
Note that we consider subtrees in a general sense. (Thus, e.g., not just
fringe trees; for them, see similar results in \cite{SJ285}.)

The purpose of the present paper is to show the following law of large
numbers under minimal moment assumptions.
Let  $\nut(T)$ be the number of rooted copies of $\ttt$ in a tree $T$, 
\ie, copies with the root at the root of $T$. Further, let $\gdt$ be the
maximum outdegree in $\ttt$.

\begin{theorem}\label{T1}
  Let $\ttt$ be a fixed ordered tree, and let $\ctn$ be a \cGWt{} defined by an
  offspring distribution $\xi$ with $\E\xi=1$ and $\E\xi^{\gdt}<\infty$.
Also, let $\cT$ be a \GWt{} with the same offspring distribution.
Then, as \ntoo,
\begin{align}
  \label{t1}
\ntt(\ctn)/n \lito \E\nut(\cT),
\end{align}
where the limit is finite  and given explicitly by \eqref{l1a} below.

Equivalently,
\begin{align}
  \ntt(\ctn)/n &\pto \E\nut(\cT),\label{t1p}                   
\end{align}
and
\begin{align}
\E\ntt(\ctn)/n &\to \E\nut(\cT).\label{t1e}
\end{align}
\end{theorem}

The fact that \eqref{t1} is equivalent to \eqref{t1p}--\eqref{t1e} is an
instance of the general fact that for any random variables,
convergence in $L^1$ is equivalent to convergence in probability together
with convergence of the means of the absolute values (\ie, in this case, with
non-negative variables, the means); see \eg{} \cite[Theorem 5.5.4]{Gut}.
We nevertheless state both versions for convenience.

\citet{Drmotaetal} (see also \cite[Section 3.3]{Drmota})  considered 
patterns in random trees;
their patterns differ from the subgraph counts above in that
some external vertices are added to $\ttt$, and that one only considers copies
of $\ttt$ in a tree $T$ such that each internal vertex in the copy
has the same degree
in $T$ as in $\ttt$ (counting also edges to external vertices);
equivalently,
each vertex in $\ttt$ is equipped with a number, and one considers only
copies of $\ttt$ where the vertex degrees match these numbers.
(Another difference is that \cite{Drmotaetal} consider unrooted trees, but
the proof proceeds by first considering rooted [planted] trees.
Furthermore, only uniformly random labelled trees are considered in
\cite{Drmotaetal}, but the
proofs extend to suitable more general \cGWt{s}, as remarked in
\cite{Drmotaetal} 
and shown explicitly in \cite{Kok1,Kok2}.)
It was shown in \citet{Drmotaetal} that the number of occurences of such a
pattern  is asymptotically normal, with asymptotic mean and variance both of
the order $n$ (except that the variance might be smaller in at least one
exceptional degenerate case),
which of cource entails a law of large
numbers. 
Moreover, \cite{Drmotaetal} discuss briefly generalizations, including
subtrees without further degree conditions as in the present paper; they
expect asymptotic normality to hold in this case too, but it seems that
their method, which is based on setting up and analyzing a system of
functional equations for generating functions, 
in general would require extensions to infinite systems, which as far as we
know has not been pursued. (See \cite{DrmotaRR} for a related problem.)
See further \refS{Sfurther}.

Our method is probabilistic, and quite different from the analysis of
generating functions in \cite{Drmotaetal}.

\section{Notation}\label{Snot}

All trees are rooted and ordered. The root of a tree $T$ is denoted $o=o_T$.
The size $|T|$ of a tree $T$ is defined as the number of vertices in $T$.

The \emph{degree} $d(v)$ of a vertex $v\in T$ always means the outdegree, \ie,
the number of children of $v$.
The \emph{degree sequence} of $T$ is the sequence of all degrees $d(v)$,
$v\in T$, for definiteness in \dfo.
Let $\gD(T):=\max_{v\in T}d(v)$ be the maximum (out)degree in $T$.

A (general) \emph{subtree} $T'$ of a tree $T$ is a non-empty connected
subgraph of $T$;
we regard a subtree as a rooted tree in the obvious way, with the root being
the vertex in $T'$ that is closest to the root in $T$.
Note that for any vertex $v\in T'$, its set of children in $T'$ is a
subset of its set of children in $T$; the order of the children of $v$ in
$T'$ is 
(by definition) the same as their relative order in $T$.

If $v\in T$, the \emph{fringe subtree} $T^v$ is the subtree of $T$ consisting
of $v$ and all its descendants; this is thus a subtree with root $v$.

If $\ttt$ and $T$ are ordered rooted tree, let $\ntt(T)$ be the number of (general)
subtrees of $T$ that are isomorphic to $\ttt$ (as ordered trees),
and let $\nut(T)$ be the number of such subtrees that furthermore have root
$o_T$. Then $\nut(T^v)$ is the number of subtrees with root $v$ 
isomorphic to $\ttt$, and thus
\begin{align}\label{ntt}
  \ntt(T) = \sum_{v\in T}\nut(T^v).
\end{align}
In other words, $\ntt(T)$ is an additive functional with toll function
$\nut(T)$, see \eg{} \cite{SJ285}.

Let $\cT$ be a random \GWt{} defined by an offspring distribution 
$(p_i)\xooo$,
and let $\ctn$ be the \cGWt{} defined as $\cT$ conditioned on $|\cT|=n$
(tacitly considering only $n$ such that $\P\bigpar{|\cT|=n}>0$); see \eg{}
\cite{SJ264} for a survey.
We let $\xi$ be a random variable with the distribution $(p_i)\xooo$;
we call both $(p_i)\xooo$ and (with a minor abuse) $\xi$ the \emph{offspring
  distribution}.
We will only consider offspring distributions with $\E\xi=1$ (\ie,
$\xi$ is \emph{critical}). 
(We often repeat this for emphasis.)
Let $\gss:=\Var\xi\le\infty$; we tacitly assume $\gss>0$, 
but do not require $\gss<\infty$ unless we say so.

$C$ and $c$ denote unspecified constants that may vary from one occurrence
to the next. They may depend on parameters such as the offspring
distribution or the fixed tree $\ttt$, but they never depend on $n$.

Convergence in probability and distribution is denoted $\pto$ and $\dto$,
respectively. 
Unspecified limits are as \ntoo.

\section{Proof}\label{Spf}

We begin by finding the expectation of $\nut$ for both unconditioned and
conditioned \GWt{s}.
Let 
\begin{align}
  S_n:=\sumin \xi_i,
\end{align}
where $\xi_1,\xi_2,\dots$ are \iid{} copies of $\xi$.

\begin{lemma}
  \label{L1}
Let $\ttt$ be a fixed ordered tree with  degree sequence $d_1,\dots,d_k$, 
where thus $k=|T|$. 
\begin{romenumerate}
  
\item \label{L1a}
Then
\begin{align}\label{l1a}
  \E \nut(\ct)=\prodik \E\binom \xi{d_i}
=\prodik \sum_{m_i=d_i}^\infty p_{m_i}\binom{m_i}{d_i}.
\end{align}

\item \label{L1b}
If $n>k$, then, with $m:=\sumik m_i$,
\begin{align}\label{l1b}
  \E \nut(\ctn)
=\frac{n}{n-k} \sum_{m_1,\dots m_k\ge0}\prodik  p_{m_i}\binom{m_i}{d_i}
\cdot \frac{(m-k+1)\P(S_{n-k}=n-m-1)}{\P(S_n=n-1)}.
\end{align}

\end{romenumerate}
\end{lemma}

\begin{proof}

\pfitemref{L1a}  
We try to construct a copy $t'$ of $\ttt$ in $\ct$, with the given root $o$.
Let $m_1$ be the root degree of $\ct$. Then there are $\binom{m_1}{d_1}$
ways to choose the $d_1$ children of the root that belong to $t'$.
Fix one of these choices, say $v_{11},\dots,v_{1d_1}$.

Next, let $m_2$ be the number of children of $v_{11}$ in $\ct$.
Given $m_2$, 
there are $\binom{m_2}{d_2}$
ways to choose the $d_2$ children of $v_{11}$ that belong to $t'$.
Fix one of these choices.

Continuing in the same way, taking the vertices of $t'$ in \df{} order, we
find for every sequence $m_1,\dots,m_k$ of non-negative integers,
a total of $\prod_1^k\binom{m_i}{d_i}$ choices, and each of these gives a
tree $t'\cong t$ provided the selected vertices in $\cT$ have degrees
$m_1,\dots,m_k$, which occurs with probability $\prodik p_{m_i}$.
Hence,
\begin{align}
  \E \nut(\ct)&
= \sum_{m_1,\dots m_k\ge0} \prodik p_{m_i} \prodik \binom{m_i}{d_i}
= \sum_{m_1,\dots m_k\ge0} \prodik \Bigpar{p_{m_i} \binom{m_i}{d_i}}
\notag\\&
=\prodik \sum_{m_i=0}^\infty p_{m_i}\binom{m_i}{d_i},
\end{align}
and \eqref{l1a} follows.

\pfitemref{L1b}  
Consider again $\ct$. We have just shown that
each sequence $m_1,\dots,m_k$ gives $\prodik\binom{m_i}{d_i}$
choices of possible subtrees $t'\cong t$ in $\cT$, 
where the vertices of $t'$ are
supposed to have degrees $m_1,\dots m_k$ in $\ct$.
This gives a total of $m=\sumik m_i$ children, of which $k-1$ are the non-root
vertices in $t'$, and thus $m-(k-1)$ are unaccounted children.
Then, $|\ct|=n$ if and only if these $m-k+1$ children and their descendants
yield exactly $n-k$ vertices. 

Condition on $m_1,\dots,m_k$ and one of the corresponding choices of $t'$.
The probability that the $m-k+1$ children above and their descendants are
$n-k$ vertices  is the probability that a \GWp{} (with offspring
distribution $\xi$) started witk $m-k+1$ individuals has total progeny
$n-k$, which by the Otter--Dwass formula 
\cite{Dwass}
(see also 
\cite{Pitman:enum} and the further references there)
is given by 
\begin{align}
  \frac{m-k+1}{n-k}\P\bigpar{S_{n-k}=n-k-(m-k+1)}.
\end{align}
Multiplying with $\prodik p_{m_i}$, the probability that the vertices in
$t'$ have the right degrees in $\ct$, 
and summing over all possibilities, we obtain
\begin{align}\label{lab}
&  \E\bigsqpar{ \nut(\ctn)}\P\bigpar{|\ct|=n}
=
 \E\bigsqpar{\nut(\ct)\mid |\ct|=n}\P\bigpar{|\ct|=n}
=  \E\bigsqpar{\nut(\ct)\indic{|\ct|=n}}
\notag\\&\qquad
= \sum_{m_1,\dots m_k\ge0}\prodik  p_{m_i}\binom{m_i}{d_i}
\cdot \frac{m-k+1}{n-k}\P(S_{n-k}=n-m-1)
.\end{align}
By the Otter--Dwass formula again 
(this time the original case in \cite{Otter}),
\begin{align}
  \P\bigpar{|\ct|=n}
=
\frac{1}n P\bigpar{S_n=n-1}
\end{align}
and \eqref{l1b} follows.
(Cf.\ \cite[Lemma 15.9]{SJ264} for a related result.)
\end{proof}

We need estimates of the probabilities 
$\P\bigpar{S_n=n-m}$. The estimate \eqref{erika0} below is standard;
we expect that also \eqref{erika1} is known, but we have not found a
reference, so we give a proof. (It is related to more difficult estimates in
\eg{} \cite{Petrov} assuming more moments, see \refR{RP} below.)
\begin{lemma}
  \label{Lsn}
 Suppose that $\E\xi=1$ and $\E\xi^2<\infty$.
Then, 
uniformly for all $n\ge1$ and $m\in\bbZ$, 
\begin{align}\label{erika0}
\P\bigpar{S_n=n-m} &\le C n\qqw,
\\\label{erika1}
\P\bigpar{S_n=n-m} &\le C |m|\qw.
\end{align}
\end{lemma}
\begin{proof}
\pfitem{\ref{erika0}}
This is well-known. In fact, the classical local limit theorem,
see \eg{} \cite[Theorem VII.1]{Petrov}, 
gives the much more precise result that, uniformly in $m\in\bbZ$ as \ntoo,
\begin{align}\label{local}
  \P\bigpar{S_n=n-m}
=\frac{h}{\gs\sqrt{n}}\Bigpar{\frac{1}{\sqrt{2\pi}}e^{-m^2/{2\gss n}}+o(1)}.
\end{align}
where $h$ is the span of the offspring distribution.
(Provided $h|(n-m)$; otherwise the probability is 0.)

\pfitem{\ref{erika1}}
Let $\gf(t):=\E e^{\ii t (\xi-1)}$ be the \chf{} of $\xi-1=\xi-\E\xi$; note
that $\gf(t)$ is twice differentiable because $\E\xi^2<\infty$.
Then, by Fourier inversion,
\begin{align}\label{sofie}
  \P\bigpar{S_n=n-m}
=\frac{1}{2\pi}\intpipi e^{\ii mt}\gf(t)^n\dd t.
\end{align}
Hence, using an integration by parts,
\begin{align}
2\pi\ii m  \P\bigpar{S_n=n-m}
=\intpipi\Bigpar{\ddt e^{\ii mt}}\gf(t)^n\dd t
=-\intpipi e^{\ii mt}\ddt\bigpar{\gf(t)^n}\dd t
\end{align}
and thus
\begin{align}\label{magnus}
| m|  \P\bigpar{S_n=n-m}
\le
\intpipi \Bigabs{\ddt\bigpar{\gf(t)^n}}\dd t
=
n\intpipi \abs{\gf'(t)}\abs{\gf(t)}^{n-1}\dd t.
\end{align}
The assumptions yield $\gf'(0)=\E(\xi-1)=0$ and
$\sup|\gf''(t)|=|\gf''(0)|=\Var\xi=C<\infty$, 
and thus 
\begin{align}
  \label{emma}
|\gf'(t)|\le Ct
.\end{align}
Assume for simplicity that the span of $\xi$ is 1 (the general case is
similar, with standard modifications). Then, as is well-known, it is easy to
see that there exist $c>0$ such that
\begin{align}\label{jesper}
  |\gf(t)|\le e^{-ct^2},\qquad |t|\le\pi.
\end{align}
Using \eqref{emma} and \eqref{jesper} in \eqref{magnus} we obtain
\begin{align}\label{wilhelm}
| m|  \P\bigpar{S_n=n-m}
\le
n C\intpipi |t| e^{-c(n-1)t^2}\dd t
\le Cn\intoo t e^{-cnt^2}\dd t
=C
,\end{align}
which proves \eqref{erika1}.
\end{proof}

\begin{remark}\label{RP}
  In the same way, taking two derivatives inside \eqref{sofie}, one obtains
\begin{align}\label{erika2}
\P\bigpar{S_n=n-m} &\le C n\qq m\qww,
\end{align}
which is stronger for large $m$; note that \eqref{erika0} and \eqref{erika2}
  imply \eqref{erika1}.
Furthermore, even stronger estimates hold if we assume more moments;
see \cite[Theorem VII.16]{Petrov} for a precise asymptotic estimate assuming
$\E\xi^k<\infty$ for some $k\ge3$. In fact, \cite[Theorem VII.16]{Petrov}
holds for $k=2$ too, which can be seen by refining the argument above;
this is an asymptotic estimate that is more  precise than \eqref{erika2}
(and implies it).
\end{remark}

\begin{lemma}
  \label{Lon}
Let $\ttt$ be a fixed ordered tree and suppose that $\E\xi=1$, $\E\xi^2<\infty$
and $\E\xi^{\gdt}<\infty$.
Then 
$\E\nut(\ctn)=o\bigpar{n\qq}.$
\end{lemma}

\begin{proof}
  Let again 
the degree sequence of $\ttt$ be $d_1,\dots,d_k$.
For a vector $\bm=(m_1,\dots,m_k)\in \bbzzk$, let
\begin{align}\label{amma}
  a_\bm:=
\prodik  p_{m_i}\binom{m_i}{d_i}.
\end{align}
Then, \eqref{l1a}--\eqref{l1b} and the assumption $\E\xi^\gdt<\infty$ yield
\begin{align}\label{cecilia}
  \sumbm a_\bm=\E\nut(\cT)<\infty
\end{align}
and
for $n>k$,
with as above $m:=\sum_im_i =:|\bm|$ (and $C=1$, actually),
\begin{align}\label{selma}
  \E\nut(\ctn) \le C
\sumbm a_\bm
\cdot \frac{m\P(S_{n-k}=n-m-1)}{\P(S_n=n-1)}.
\end{align}
Denote the summand in \eqref{selma} by $b_{\bm,n}$.
By the local limit theorem \eqref{local},
as is well-known,
\begin{align}
  \label{sigrid}
\P(S_n=n-1)\sim c n\qqw,
\end{align}
and thus
\begin{align}\label{winston}
  b_{\bm,n}/n\qq \le 
C  ma_\bm \P(S_{n-k}=n-m-1).
\end{align}
Hence, \eqref{erika0} implies that for every fixed $\bm$, as \ntoo,
\begin{align}\label{eleonora}
  b_{\bm,n}/n\qq \le 
C  ma_\bm n\qqw \to 0.
\end{align}
Furthermore, \eqref{winston} and \eqref{erika1} 
yield
\begin{align}\label{anna}
  b_{\bm,n}/n\qq \le 
C  a_\bm,
\end{align}
which is summable by \eqref{cecilia}.
Consequently, dominated convergence shows that
\begin{align}\label{lina}
n\qqw   \sumbm b_{\bm,n}=  \sumbm b_{\bm,n}/n\qq \to0,
\end{align}
which together with \eqref{selma} yields the result
$n\qqw\E\nut(\ctn)\to0$.
\end{proof}

We will see in \refE{Ebad} below, that the estimate $o(n\qq)$ in \refL{Lon} is
best possible in general.
However, if we assume another moment on $\xi$, we can improve the estimate
to $O(1)$, and furthermore show that $\E\nut(\ctn)$ converges. We next show
this,  although it is not required for our main result.
\begin{lemma}\label{L3}
  Let $\ttt$ be a fixed tree with degree sequence $d_1,\dots,d_k$,
and suppose that $\E\xi=1$.
Then, as \ntoo,
\begin{align}\label{l3}
  \E\nut(\ctn)\to
\sumik (d_i+1)\E\binom{\xi}{d_i+1} \prod_{j\neq i} \E\binom\xi{d_j}.
\end{align}
In particular, $\E\nut(\ctn)=O(1)$ if\/ $\E\xi^{\gdt+1}<\infty$, while
$\E\nut(\ctn)\to\infty$ if\/ $\E\xi^{\gdt+1}=\infty$.
\end{lemma}

\begin{proof}
Define again $a_\bm$ by \eqref{amma}, and
  denote the summand in \eqref{l1b} by $b'_{\bm,n}$, where as above
  $\bm=(m_1,\dots,m_k)\in\bbzzk$. 
It follows from the local limit theorem \eqref{local} that for every fixed
$\bm$, as \ntoo,
\begin{align}
  \frac{\P(S_{n-k}=n-m-1)}{\P(S_n=n-1)} 
=\frac{h\xpar{2\pi\gss( n-k)}\qqw\bigpar{1+o(1)}}
{h\xpar{2\pi\gss n}\qqw\bigpar{1+o(1)}}
\to 1.
\end{align}
(This holds also if the span $h>1$, assuming as we may that all $p_{m_i}>0$,
so $h|m$.) Hence,
\begin{align}\label{elfbrink}
  b'_{\bm,n}\to a_\bm(m-k+1).
\end{align}
Furthermore, by \eqref{erika0} and \eqref{sigrid},
\begin{align}
  \frac{\P(S_{n-k}=n-m-1)}{\P(S_n=n-1)} 
\le\frac{C n\qqw}{c n\qqw} =C,
\end{align}
and thus
\begin{align}\label{helene}
  b'_{\bm,n}\le C a_\bm(m-k+1).
\end{align}
Consequently, if $\sum_\bm a_\bm(m-k+1)<\infty$,
then
\begin{align}\label{sumb}
  \sum_\bm b'_{\bm,n}\to\sum_\bm a_\bm(m-k+1)
\end{align}
by \eqref{elfbrink}, \eqref{helene} and dominated convergence.
On the other hand, if $\sum_\bm a_\bm(m-k+1)=\infty$,
then $\sum_\bm b'_\bm\to\infty$ by \eqref{elfbrink} and Fatou's lemma, and
thus \eqref{sumb} holds in this case too.
Recalling \eqref{l1b}, this shows that in any case,
\begin{align}\label{duncan}
  \E\nut(\ctn)\to\sum_\bm a_\bm(m-k+1),
\end{align}
and it remains only to evaluate the limit.

Since $\ttt$ is a tree, we have $\sumik d_i=k-1$, and thus
$m-k+1=\sumik(m_i-d_i)$. Recalling the definition \eqref{amma} of $a_\bm$,
we thus have
\begin{align}\label{duff}
  \sum_\bm a_\bm(m-k+1) 
&=\sum_\bm\sumik (m_i-d_i) p_{m_i}\binom{m_i}{d_i}
  \prod_{j\neq i} p_{m_j}\binom{m_j}{d_j}
\notag\\&
=\sumik\sum_{m_i=0}^\infty p_{m_i}(m_i-d_i)\binom{m_i}{d_i}
  \prod_{j\neq i} \sum_{m_j=0}^\infty p_{m_j}\binom{m_j}{d_j}
,\end{align}
which equals the \rhs{} of \eqref{l3} because
$(m_i-d_i)\binom{m_i}{d_i}=(d_i+1)\binom{m_i}{d_i+1}$.
This completes the proof by \eqref{duncan}.
\end{proof}

\begin{remark}\label{Rht}
Assume only $\E\xi=1$.
If $\hT$ is the infinite size-biased \GWt{} defined by \citet{Kesten},
see also \cite[Section 5]{SJ264},
then $\ctn\dto\hT$ in a local topology (\ie, close to the root), see
\cite[Theorem 7.1]{SJ264}, and it follows that 
\begin{align}
  \label{alla}
\nut(\ctn)\dto\nut(\hT).
\end{align}
It is not difficult to see that $\E\nut(\hT)$ equals the \rhs{} of
\eqref{l3}, which thus says that $\E\nut(\ctn)\to\E\nut(\hT)$.
(This could presumably be used to give an alternative proof of \refL{L3},
but we prefer the direct proof above.)

In particular, if $\E\xi^{\gdt+1}=\infty$, then $\E\nut(\hT)=\infty$, and
thus \eqref{alla} and Fatou's lemma yield $\E\nut(\ctn)\to\infty$.
Hence, the last sentence in \refL{L3} holds also without the assumption
$\E\xi^2<\infty$. 
\end{remark}

We proceed to the proof of \refT{T1}. 
The case $\gdt\le1$ is special, since we then do not assume
$\E\xi^2<\infty$,
but on the other hand this case is simple and rather trivial, 
so we discuss it separately in the following example.
\begin{example}\label{Ep}
Consider the case $\gdt\le1$. This means that $\ttt$ is a path $\pk$
with $k\ge1 $ vertices, and thus length $k-1$.
A copy of $\ttt$ in a tree $T$ is thus a path consisting of $k$ vertices
$v_1,\dots,v_k$ such that $v_{i+1}$ is a child of $v_i$; such a path is
determined by its endpoint $v_k$, and every vertex of
depth (= distance from the root) at least $k-1$ is the endpoint of a copy of
$\ttt$. Hence, if $\nu_i(T)$ is the number of vertices in $T$ of depth $i$,
then
\begin{align}\label{epa}
  \nnx_{\pk}(T)=\sum_{i\ge k-1}\nu_i(T)
=
|T|-\sum_{i=0}^{k-2}\nu_i(T).
\end{align}
In particular, 
$\nnx_{\px1}(\ctn)=n$ and
$\nnx_{\px2}(\ctn)=n-1$ are deterministic; these are trivially
just the numbers of vertices and edges.

Moreover, as said in \refR{Rht}, assuming $\E\xi=1$, 
the random tree $\ctn$ converges locally in distribution as \ntoo,
see  \cite[Theorem 7.1]{SJ264}; in particular each $\nu_i(\ctn)$ converges in
distribution (to $\nu_i(\hT)$)
and thus $\nu_i(\ctn)=\Op(1)$ (\ie, is bounded in probability).
Hence, 
for every $k\ge1$, \eqref{epa} implies
\begin{align}\label{epb}
\nnx_{\pk}(\ctn)=n+\Op(1).
\end{align}
In particular,
$\nnx_{\pk}(\ctn)$ is more strongly concentrated than the dispersion of
order $n\qq$ typically seen in similar statistics, see \eg{} \refE{E11} and
\refS{Sfurther}.
\end{example}

\begin{proof}  [Proof of \refT{T1}]
Suppose first $\gdt\le1$.
Then $\ttt=\pk$ for some $k\ge1$ and
\refE{Ep} shows that \eqref{epb} holds, and thus
$\nnx_\pk(\ctn)/n\pto1$.
Furthermore, \eqref{l1a} yields 
\begin{align}
  \E\nux_{\pk}(\ct)=(\E \xi)^{k-1}=1, 
\end{align}
and thus \eqref{t1p} holds. Moreover, $\nnx_\pk(\ctn)/n\le1$ by \eqref{epa},
and thus dominated convergence applies to \eqref{t1p} 
and yields \eqref{t1e} and \eqref{t1}, 
see \eg{} \cite[Theorems 5.5.4 and 5.5.5]{Gut}.  

In the remainder of the proof we may thus assume $\gdt\ge2$, and thus, in
particular, $\E\xi^2<\infty$. 
(The arguments below use $\E\xi^2<\infty$, but apply to any $\gdt$.)

\refL{L1}\ref{L1a} and the assumption $\E\xi^{\gdt}<\infty$ show that
$\E\nut(\ct)<\infty$, and \refL{Lon} shows $\E\nut(\ctn)=o\bigpar{n\qq}$.
Hence \eqref{t1p} and \eqref{t1e} follow by \cite[Remark 5.3]{SJ285}.
However, since only a sketch of the proof is given in that remark, let us
add some details.

First, 
\eqref{t1e} follows by the
argument in the proof of \cite[Theorem 1.5(i)]{SJ285}, adding the factor
$n\qq$ at some places.

Next, define for $M>0$ the truncation $\nutm(T):=\nut(T)\bmin M$ and let
$\nttm(T):=\sum_{v\in T} \nutm(T^v)$ be the corresponding additive
functional, \cf{} \eqref{ntt}.
Let $\eps>0$. Since $\nutm(\cT)\upto\nut(\ct)$ as $M\to\infty$,
we can by monotone convergence, and $\E\nut(\ct)<\infty$,
choose $M$ such that 
\begin{align}\label{othello}
\E\nut(\cT)-\E\nutm(\ct)<\eps^2.
\end{align}
We have proved \eqref{t1e}, and similarly $\E\nttm(\ctn)/n\to\E\nutm(\cT)$ by
\cite[Theorem 1.3]{SJ285}, since $\nutm$ is bounded. Hence, \eqref{othello}
implies that for all
sufficiently large $n$, 
\begin{align}\label{desdemona}
\E\bigabs{\ntt(\ctn)/n-\nttm(\ctn)/n}=
\E\ntt(\ctn)/n-\E\nttm(\ctn)/n<\eps^2.
\end{align}
Furthermore, \cite[Theorem 1.3]{SJ285} also yields
$\nttm(\ctn)/n\pto\E\nutm(\cT)$.
Consequently, using also \eqref{othello} again, \eqref{desdemona}
 and Markov's inequality, if
$n$ is large,
\begin{align}
&  \P\bigpar{\bigabs{\ntt(\ctn)/n-\E\nut(\ct)}>3\eps}
\notag\\&
\le
  \P\bigpar{\bigabs{\ntt(\ctn)/n-\nttm(\ctn)/n}>\eps}
+  \P\bigpar{\bigabs{\nttm(\ctn)/n-\E\nutm(\ct)}>\eps}
\notag\\&
\le2\eps.
\end{align}
Hence, \eqref{t1p} holds.

Finally, as said earlier, \eqref{t1p} and \eqref{t1e} are together
equivalent to the $L^1$ convergence \eqref{t1}.
\end{proof}

\section{Examples}\label{Sex}
We give some simple but illuminating examples.
Recall also \refE{Ep}.

\begin{example}\label{Eqr}
  Let $t=t_{q,r}$ consist of two paths with $q+1$ and $r+1$ vertices, joined
  at the 
  root; here $q,r\ge1$. We have $k=1+q+r$ and $d_1=2$ while $d_i=1$ for
  $i>1$; thus $\gdt=2$.
Since $\E\xi=1$, \eqref{l1a} yields
\begin{align}
  \E\nuqr(\ct)=\E\binom{\xi}2=\frac{\E\xi^2-1}{2}=\frac{\gss}2.
\end{align}
Hence, \refT{T1} yields, for any $q,r\ge1$,
\begin{align}\label{eqr2}
  \nnqr(\ctn)/n\lito \gss/2.
\end{align}
\end{example}

\begin{example}\label{E11}
Consider the special case $q=r=1$ of \refE{Eqr}. Then $t_{1,1}$ is a
cherry, \ie, a root with two children.
If a vertex $v$ in a tree $T$ has degree $d(v)$, then the number of cherries
rooted at $v$ 
is $\binom{d(v)}2$, and thus
\begin{align}\label{e11}
  \nnx_{t_{1,1}}(T)=\sum_{v\in T}\binom{d(v)}2
=\sumr \binom r2 X_r(T),
\end{align}
where $X_r(T)$ is the number of vertices of degree $r$ in $T$.

It is known that $X_r(\ctn)/n\pto p_r$, see \eg{} \cite[Theorem 7.11]{SJ264}.
Hence, \eqref{eqr2} (with $q=r=1$)
is what we would get by dividing \eqref{e11} by $n$ and
taking the limit inside the sum; if the degree distribution is bounded, the
sum is finite so this is rigorous and \eqref{eqr2} (still with $q=r=1$)
follows from \eqref{e11}.

In this case we can say much more than \eqref{eqr2}.
 It was proved in \cite{Kolchin},
see also \cite{DrmotaG99},
that $X_r(\ctn)$ is asymptotically
normal, with
\begin{align}
  \frac{X_r(\ctn)-np_r}{\sqrt n}\dto N\bigpar{0,\gam_r^2}
\end{align}
for some explicit $\gam_r^2$.
This was extended to joint convergence for all $r$ in \cite{SJ132}, provided
$\E\xi^3<\infty$. Hence, at least if $\xi$ is bounded, it follows from
\eqref{e11} that
$\nnx_{t_{1,1}}(\ctn)$ is asymptotically
normal, with
\begin{align}\label{asn}
  \frac{\nnx_{t_{1,1}}(\ctn)-n\gss/2}{\sqrt n}\dto N\bigpar{0,\gam^2}
\end{align}
for some explicit $\gam^2\ge0$.
There are degenerate cases where $\gam^2=0$. For example, for full binary
trees ($\P(\xi=2)=\P(\xi=0)=\frac12$), all degrees are 0 or 2, and then
each $X_r(T)$ is a deterministic function of $|T|$; hence \eqref{e11} shows
that $\nnx_{t_{1,1}}(\ctn)$ is deterministic. 
More generally, the same happens for full
$m$-ary trees, with $\xi\in\set{0,m}$ \as, for any $m\ge2$. But 
it can be seen from the covariances given in \cite{SJ132} that $\gam^2>0$ in all
other cases with bounded $\xi$.
See further \refS{Sfurther}.
\end{example}

\begin{example}
  \label{El}
Let $\ell\ge1$, and let $\vvl(T)$ be the number of (undirected) paths of
length $\ell$ in $T$. 
For definiteness, we count undirected paths, so this equals the number of
unordered pairs $(v,w)$ of vertices of distance $\ell$.
There are two cases: 
\begin{romenumerate}
  
\item 
$v$ is an ancestor of $w$, or conversely; the number of such pairs is
$\nnx_{\pl}(T)$.

\item 
Neither $v$ nor $w$ is an ancestor of the other. 
Then $v$ and $w$ are the two leaves in a copy of $\tqr$ 
with $q,r\ge1$ and $q+r=\ell$.
For given $q$ and $r$, the number of such pairs equals $\nnqr(T)$
\end{romenumerate}
Consequently,
\begin{align}
  \vvl(T)=\nux_{\pl}(T)+\sum_{q=1}^{\ell-1}\nnx_{t_{q,\ell-q}}(T).
\end{align}
Hence, \refEs{Ep} and \ref{Eqr} yield
\begin{align}\label{julie}
  \vvl(\ctn)/n\lito 1+(\ell-1)\frac{\gss}2.
\end{align}

For example, taking $\xi\sim\Po(1)$ we obtain (forgetting the ordering)
a uniformly random unordered labelled tree; 
we have $\gss=1$ and thus \eqref{julie} yields 
\begin{align}\label{vvpo}
  \vvl(\ctn)\lito (\ell+1)/2.
\end{align}

Similarly, taking $\xi\sim\Ge(1/2)$ we obtain
a uniformly random ordered  tree; 
we have $\gss=2$ and thus \eqref{julie} then yields 
\begin{align}\label{vvge}
  \vvl(\ctn)\lito \ell.
\end{align}

 Taking $\xi\sim\Bi(2,1/2)$ we obtain
a uniformly random binary tree; 
we have $\gss=1/2$ and thus \eqref{julie} now yields 
\begin{align}\label{vvge}
  \vvl(\ctn)\lito (\ell+3)/4.
\end{align}

\end{example}

The following example shows that the estimate $o\bigpar{n\qq}$ in \refL{Lon}
is best possible.
\begin{example}
  \label{Ebad}
For simplicity,  let the tree $\ttt$ be a star, where the root has degree
$\gD\ge2$ and its children are leaves with degree 0. (The argument is
easily modified to any tree $\ttt$ with $\gdt\ge2$.)
Thus $k:=|t|=\gD+1$.
Assume that the span of $\xi$ is 1.

The local limit theorem \eqref{local} implies that if 
$n$ is large and $m\le n\qq$, 
then  
\begin{align}
\P(S_{n-k}=n-m-1)\ge cn\qqw  
,\end{align}
and thus, using \eqref{sigrid},
\begin{align}
\P(S_{n-k}=n-m-1)/\P(S_n=n-1)\ge c.  
\end{align}
Hence, by  \eqref{l1b} and considering there only terms with $m_2=\dots=m_k=0$,
\begin{align}\label{nutt}
  \E\nut(\ctn)
\ge
c\sum_{\gD<m_1\le n\qq} p_{m_1}\binom{m_1}{\gD} m_1
\ge
c\sum_{\gD<m\le n\qq} p_{m}{m}^{\gD+1}.
\end{align}
If $\eps>0$, and we let $p_m=m^{-\gD-1-\eps}$ for large $m$, then
$\E\xi^\gD<\infty$, and \eqref{nutt} yields, for large $n$,
\begin{align}\label{nutts}
  \E\nut(\ctn)
\ge
c\sum_{\gD<m\le n\qq} {m}^{-\eps}
\ge c n^{(1-\eps)/2}.
\end{align}
Hence,  for any $\eps>0$, $\E\nut (\ctn)$ can grow faster than
$n^{1/2-\eps}$.

Similarly, we can find an offspring distribution $(p_m)\xooo$
satifying the conditions such that $\E\nut(\ctn)=n^{1/2-o(1)}$; we omit the
details. 
Moreover, for any given sequence $\gd(n)\downto0$, we can find  $(p_m)\xooo$
such that 
$\E\nut(\ctn)\ge\gd(n)n\qq$, at least for a subsequence.  To see this, take
an increasing sequence $(m_j)\xoo$ with $\sumj j\gd(m_j^2)<1$.
Let $p_{m_j}:=j\gd(m_j^2)m_j^{-\gD}$, and $p_m=0$ for all other $m\ge2$,
choosing $p_0$ and $p_1$ such that $\sum_i p_i=\sum_ii p_i=1$.
Also, let $n_j:=m_j^2$. Then \eqref{nutt} implies that, for large $j$,  
\begin{align}
  \E\nut(\ct_{n_j}) \ge c p_{m_j}m_j^{\gD+1} 
=cj m_j \gd(m_j^2)
\ge m_j \gd(m_j^2)
= n_j\qq \gd(n_j) .
\end{align}
\end{example}

\section{Asymptotic normality?}\label{Sfurther}

We showed in \refE{E11} that if $\xi$ is bounded, then
$\nnx_{t_{1,1}}(\ctn)$ is asymptotically normal in the sense that
\eqref{asn} holds (although $\gam^2=0$ is possible). 
In fact, this holds for any fixed tree $\ttt$.

\begin{proposition}\label{P5}
  Assume that $\xi$ is bounded. Then, for any fixed tree $\ttt$, 
\begin{align}\label{asn5}
  \frac{\ntt(\ctn)-n\mu_\ttt}{\sqrt n}\dto N\bigpar{0,\gam_\ttt^2},
\end{align}
for $\mu_\ttt:=\E\nut(\ct)$ and some $\gam^2_\ttt\ge0$.
\end{proposition}
\begin{proof}
  This follows from the result by \citet{Drmotaetal} on patterns discussed
  in \refS{S:intro}  (extended to \cGWt{s} \cite{Drmotaetal,Kok1,Kok2}); 
the assumption on $\xi$ means that vertex degrees are
  bounded by some constant, and thus there is a finite number of patterns
  that correspond to subtrees isomorphic to $\ttt$; hence $\ntt(\ctn)$ is a
  linear combination of pattern counts, and the result follows from the
  joint asymptotic normality of the latter.
(See also \cite{LiLi} for a special case.)

Alternatively, this is an  application of \cite[Theorem 1.13]{SJ285}:
the functional $\nut$ is local (as defined in \cite{SJ285}) and for trees
with degrees bounded by some constant $K$, $\nut$ is bounded. 
Hence
\eqref{asn5} follows from \cite[Theorem 1.13]{SJ285}.
\end{proof}

We conjecture that this behaviour is typical, and that \refProp{P5} 
holds for every $\xi$ with $\E\xi=1$ that satisfies a suitable moment
condition. 
However, it seems that substantial additional work would be required to show
this. As said in the introduction,  this was briefly discussed in
\cite{Drmotaetal}, but it seems that the method there requires extensions to
infinite systems of functional equations.
Similarly, the application of \cite[Theorem 1.13]{SJ285} requires
$\nut(\ctn)$ to be bounded, which is not the case when $\xi$ is unbounded.
It is possible that this may be overcome by truncations and some variance
estimates, 
but again more work is needed.
(The extension in \cite{WagnerRS2020} applies to 
the case when $\ttt$ is a star with root degree $\gD$ (including
\refE{E11} with $\gD=2$) and $\E\xi^{2\gD+1}<\infty$;
this might suggest further extensions.)
This problem is thus left for future research.

Note also that there are degenerate cases when the asymptotic variance
  in \eqref{asn5}  $\gam^2_\ttt=0$; see \refEs{Ep} and \ref{E11}.
(Then \eqref{asn5} does not give
  asymptotic normality; only a concentration result.)
However, we conjecture that this is an exception, occuring only in a few
special cases.

\subsection*{Acknowledgement}
I thank Stephan Wagner for helpful comments.

\newcommand\AAP{\emph{Adv. Appl. Probab.} }
\newcommand\JAP{\emph{J. Appl. Probab.} }
\newcommand\JAMS{\emph{J. \AMS} }
\newcommand\MAMS{\emph{Memoirs \AMS} }
\newcommand\PAMS{\emph{Proc. \AMS} }
\newcommand\TAMS{\emph{Trans. \AMS} }
\newcommand\AnnMS{\emph{Ann. Math. Statist.} }
\newcommand\AnnPr{\emph{Ann. Probab.} }
\newcommand\CPC{\emph{Combin. Probab. Comput.} }
\newcommand\JMAA{\emph{J. Math. Anal. Appl.} }
\newcommand\RSA{\emph{Random Structures Algorithms} }
\newcommand\DMTCS{\jour{Discr. Math. Theor. Comput. Sci.} }

\newcommand\AMS{Amer. Math. Soc.}
\newcommand\Springer{Springer-Verlag}
\newcommand\Wiley{Wiley}

\newcommand\vol{\textbf}
\newcommand\jour{\emph}
\newcommand\book{\emph}
\newcommand\inbook{\emph}
\def\no#1#2,{\unskip#2, no. #1,} 
\newcommand\toappear{\unskip, to appear}

\newcommand\arxiv[1]{\texttt{arXiv}:#1}
\newcommand\arXiv{\arxiv}

\def\nobibitem#1\par{}

\end{document}